\documentclass{svproc}
\usepackage{url}
\usepackage{amsmath}
\usepackage{amssymb}
\usepackage{enumitem}
\newtheorem{observation}[theorem]{Observation}

\def\str#1{\mathbf {#1}}
\def\Alphabet{\Sigma}
\def\rel#1#2{R_{\mathbf{#1}}^{#2}}
\def\VSpace#1#2#3{\left[#1\right]^*\!{\binom{#2}{#3}}}
\def\Space#1#2#3{\left[#1\right]{\binom{#2}{#3}}}

\begin{document}
\mainmatter              
%
\title{Big Ramsey degrees and forbidden cycles}

%
\titlerunning{Big Ramsey degrees and forbidden cycles}
%
\author{Martin Balko\inst{1} \and David Chodounsk\'y\inst{1} \and Jan Hubi\v cka\inst{1}
Mat\v ej Kone\v cn\'y\inst{1} \and\\ Jaroslav Ne\v set\v ril\inst{2} \and Llu\'{i}s Vena\inst{3}}
\authorrunning{Martin Balko et al.} 
%
%
\institute{Department of Applied Mathematics (KAM), Charles University, Ma\-lo\-stranské~nám\v estí 25, Praha 1, Czech Republic,\\
\email{\{balko,chodounsky,hubicka,matej\}@kam.mff.cuni.cz},
\and
Computer Science Institute of Charles University (I\'UUK), Charles University, Ma\-lo\-stranské~nám\v estí 25, Praha 1, Czech Republic,\\
\email{nesetril@iuuk.mff.cuni.cz},
\and
Universitat Polit\`ecnica de Catalunya, Barcelona, Spain,\\
\email{lluis.vena@gmail.com}
}
\maketitle              

\begin{abstract}
Using the Carlson--Simpson theorem, we give a new general condition for a structure in a finite binary relational language to have finite big Ramsey degrees.
\end{abstract}
\section{Introduction}
We consider standard model-theoretic (relational) structures in finite binary
languages formally introduced below. Such structures may be equivalently seen as
edge-labelled digraphs with finitely many labels, however the notion of
structures is more standard in the area. Structures may be finite or countably infinite.
Given structures $\str{A}$ and $\str{B}$, we denote by $\binom{\str{B}}{\str{A}}$ the set
of all embeddings from $\str{A}$ to $\str{B}$. We write $\str{C}\longrightarrow (\str{B})^\str{A}_{k,l}$ to denote the following statement:
for every colouring $\chi$ of $\binom{\str{C}}{\str{A}}$ with $k$ colours, there exists
an embedding $f\colon\str{B}\to\str{C}$ such that $\chi$ does not take more than $l$ values on $\binom{f(\str{B})}{\str{A}}$.
For a countably infinite structure $\str{B}$ and its finite substructure $\str{A}$, the \emph{big Ramsey degree} of $\str{A}$ in $\str{B}$ is
the least number $l'\in \mathbb N\cup \{\infty\}$
such that $\str{B}\longrightarrow (\str{B})^\str{A}_{k,l'}$
for every $k\in \mathbb N$.
We say that the \emph{big Ramsey degrees of~$\str{B}$ are finite} if for every finite substructure $\str{A}$ of~$\str{B}$ the big Ramsey degree of $\str{A}$ in $\str{B}$
is finite.

We focus on structures in binary languages $L$
and adopt
some graph-theoretic terminology.  Given a structure $\str{A}$ and distinct vertices $u$ and $v$, we
say that $u$ and $v$ are \emph{adjacent} if there exists $\rel{}{}\in L$ such
that either $(u,v)\in \rel{A}{}$ or $(v,u)\in \rel{A}{}$.  A structure $\str{A}$
is \emph{irreducible} if any two distinct vertices are adjacent. 
A sequence $v_0,v_1,\ldots,v_{\ell-1}$, $\ell\geq 3$, of distinct vertices of a structure $\str{A}$
is called a \emph{cycle} if $v_i$ is adjacent to $v_{i+1}$ for every
$i \in \{0,\dots,\ell-2\}$ as well as $v_0$ adjacent to $v_\ell$.
A cycle is \emph{induced} if none of the other remaining pairs of vertices in the
sequence is adjacent.

Following ~\cite[Section 2]{Hubicka2016}, we call 
a homomorphism $f\colon \str{A}\to\str{B}$ (see Section~\ref{sec:preliminaries}) a \emph{homomorphism-embe\-dding}
if $f$ restricted to any irreducible substructure of $\str{A}$  is an
embedding. 
The homomorphism-embedding $f$ is called a \emph{(strong) completion of $\str{A}$ to $\str{B}$} provided 
that
$\str{B}$ is irreducible and $f$ is injective.

Our main result, which applies techniques developed by the third author in~\cite{Hubicka2020CS}, gives the following condition for a given structure to
have finite big Ramsey degrees.

\begin{theorem}
\label{thm:cycles}
Let $L$ be a finite language consisting of unary and binary symbols, and let $\str{K}$ be a countably-infinite irreducible structure. Assume that every countable structure $\str{A}$ has a completion to $\str{K}$ provided that
every induced cycle in $\str{A}$ (seen as a substructure) has a completion to $\str{K}$ and every irreducible substructure of $\str{A}$ of size at most $2$ embeds into $\str K$.
Then $\str{A}$ has finite big Ramsey degrees.
\end{theorem}

This can be seen as a first step towards a structural condition implying bounds
on big Ramsey degrees, giving a strengthening of results by Hubi\v cka and Ne\v set\-\v ril~\cite{Hubicka2016} to countable structures.

The study of big Ramsey degrees originates in the work of Laver who in 1969
showed that the big Ramsey degrees of the ordered set of rational numbers are finite~\cite[Chapter~6]{todorcevic2010introduction}.
The whole area has been revitalized recently; see~\cite{dobrinen2017universal,Hubicka2020CS} for references.
Our result can be used to identify many new examples of structures with finite
big Ramsey degrees. Theorem 1 is particularly fitting to examples  involving metric spaces. 
In particular, the following corollary may be of special interest.

\begin{corollary}\label{cor}
The following structures have finite big Ramsey degrees:
\begin{enumerate}[label=(\roman*)]
 \item\label{C1} Free amalgamation structures described by forbidden triangles,
 \item\label{C2} $S$-Urysohn space for finite distance sets $S$ for which $S$-Urysohn space exists,
 \item\label{C3} $\Lambda$-ultrametric spaces for a finite distributive lattice $\Lambda$~\cite{Sam},
 \item\label{C4} metric spaces associated to metrically homogeneous graphs of a finite diameter from Cherlin's list with no Henson constraints~\cite{Cherlin2013}.
\end{enumerate}
\end{corollary}

Vertex partition properties of Urysohn spaces were extensively studied in connection to oscillation stability~\cite{The2010} and determining their big Ramsey degrees presented a long standing open problem: 
Corollary~\ref{cor}
\ref{C1} is a special case of the main result of~\cite{zucker2020}, 
 \ref{C2} is a strengthening of~\cite[Corollary 6.5 (3)]{Hubicka2020CS},
 \ref{C3} strengthens~\cite{NVT2009} and~\ref{C4} is a  strengthening of~\cite{Aranda2017} to infinite structures.

To see these connections, observe that a metric space can be also represented as an irreducible structure in a binary language
having one relation for each possible distance. 
Possible obstacles to completing a structure in this language to a metric space are irreducible substructures with at most 2 vertices
and induced \emph{non-metric cycles}.  These are cycles with the longest edge of a length
exceeding the sum of the lengths of all the remaining edges; see~\cite{Aranda2017}. 

Note that all these proofs may be modified to yield Ramsey classes of finite structures. Thus, for example, \ref{C2} generalizes \cite[Section~4.3.2]{Hubicka2016}.  

Our methods yield the following common strengthening of the main results from~\cite{zucker2020} and Theorem~\ref{thm:cycles}. To obtain this result, which is going to appear in~\cite{Balko2021}, we found a new strengthening of the dual Ramsey theorem.

\begin{theorem}
\label{thm:cycles3}
Let $L$ be a finite language consisting of unary and binary symbols, and let $\str{K}$ be a countably-infinite irreducible structure. Assume that
there exists $c>0$ such that every countable structure $\str{A}$ has a completion to $\str{K}$ provided that
every induced cycle in $\str{A}$ has a completion to $\str{K}$ and every irreducible substructure of $\str{A}$ of size at most $c$ embeds into $\str K$.
Then $\str{A}$ has finite big Ramsey degrees.
\end{theorem}

\section{Preliminaries}
\label{sec:preliminaries}
A relational language $L$ is a collection of (relational) symbols $\rel{}{}\in L$, each having its {\em
arity}.  An \emph{$L$-structure} $\str{A}$ on $A$ is a structure with the {\em
vertex set} $A$ and with relations $\rel{A}{}\subseteq A^r$ for every symbol
$\rel{}{}\in L$ of arity $r$.  If the set $A$ is finite, then we call $\str A$ a
\emph{finite structure}. 
A \emph{homomorphism} $f\colon\str{A}\to \str{B}$ is a mapping $f\colon A\to B$ such that
for every $\rel{}{}\in L$ of arity $r$ we have
$(x_1,x_2,\ldots, x_{r})\in \rel{A}{}\implies (f(x_1),f(x_2),\ldots,f(x_r))\in \rel{B}{}$. A homomorphism
$f$ is an \emph{embedding} if $f$ is injective and the implication above is an equivalence.
If the identity is an embedding $\str{A}\to\str{B}$, then we call $\str{A}$ a \emph{substructure} of $\str{B}$. 
In particular, our substructures are always induced.

\medskip 

Hubi\v cka \cite{Hubicka2020CS} connected big Ramsey degrees to an infinitary dual Ramsey theorem for parameters spaces. We now review the main notions used.
Given a finite alphabet $\Alphabet$ and $k\in \omega\cup \{\omega\}$, a \emph{$k$-parameter word} is a (possibly infinite) string $W$ in the
alphabet $\Alphabet\cup \{\lambda_i\colon 0\leq i<k\}$ containing all symbols ${\lambda_i : 0\leq i < k}$ such that, for every $1\leq j < k$, the first
occurrence of $\lambda_j$ appears after the first occurrence of $\lambda_{j-1}$.
The symbols $\lambda_i$ are called \emph{parameters}.
Given a parameter word $W$, we denote its \emph{length} by $|W|$. The letter (or parameter) on index $j$ with $0\leq j < |W|$ is denoted by $W_j$. Note that the first letter of $W$ has index $0$.
A $0$-parameter word is simply a \emph{word}. 
Let $W$ be an $n$-parameter word and let $U$ be a parameter word of length $k\leq n$, where $k,n\in \omega\cup\{\omega\}$. Then 
$W(U)$ is the parameter word created by \emph{substituting} $U$ to $W$. More precisely, $W(U)$ is created from~$W$ by replacing each occurrence of $\lambda_i$, $0\leq i < k$, by $U_i$ and truncating it just
before the first occurrence of $\lambda_k$ in $W$.
Given an $n$-parameter word $W$ and a set $S$ of parameter words of length at most $n$, we define $W(S):=\{W(U)\colon U\in S\}$.

We let $\Space{\Alphabet}{n}{k}$ be the set of all $k$-parameter words of
length $n$, where $k\leq n\in \omega\cup \{\omega\}$. If $k$ is finite, then we also define $\VSpace{\Alphabet}{\omega}{k}:=\bigcup_{k\leq i < \omega}\: \Space{\Alphabet}{i}{k}$.
For brevity, we put  $\Alphabet^*:=\VSpace{\Alphabet}{\omega}{0}$.

Our main tool is the following infinitary dual Ramsey theorem, which is a special case of the Carlson--Simpson theorem \cite{carlson1984,todorcevic2010introduction}.  

\begin{theorem}
\label{thm:multCS}
Let $k\geq 0$ be a finite integer.
If $\VSpace{\emptyset}{\omega}{k}$ is coloured by finitely many colours, then there exists
$W\in \Space{\emptyset}{\omega}{\omega}$ such that $W\left(\VSpace{\emptyset}{\omega}{k}\right)$ is monochromatic.
\end{theorem}
\begin{definition}[\cite{Hubicka2020CS}]
\label{def:envelope}
Given a finite alphabet $\Alphabet$, a finite set $S\subseteq \Sigma^*$ and $d>0$, we call $W\in \VSpace{\emptyset}{\omega}{d}$ a \emph{$d$-parametric envelope} of $S$ if there exists a set $S'\subseteq
\Sigma^*$ satisfying $W(S')=S$.
In such case the set $S'$ is called the \emph{embedding type} of $S$ in $W$ and is denoted by $\tau_W(S)$.
If $d$ is the minimal integer for which a $d$-parameter envelope $W$ of $S$ exists, then we call $W$ a \emph{minimal envelope}.
\end{definition}
\begin{proposition}[\cite{Hubicka2020CS}]
\label{prop:subspace}
Let $\Alphabet$ be a finite alphabet and let $k\geq 0$ be a finite integer. Then there exists a finite $T=T(|\Alphabet|,k)$ such that
every set $S\subseteq \Sigma^*$, $|S|=k$, has a $d$-parameter envelope with $d\leq T$.
Consequently, there are only finitely many embedding types of sets of size $k$ within their corresponding minimal envelopes.
Finally, for any two minimal envelopes $W$, $W'$ of $S$, we have $\tau_{W}(S)=\tau_{W'}(S)$.
\end{proposition}

We will thus also use $\tau(S)$ to denote the type $\tau_W(S)$ for some minimal $W$.

\section{Proof of Theorem~\ref{thm:cycles}}

The proof is condensed due to the space limitations,
but we believe it gives an idea of 
fine interplay
of all building blocks.
Throughout this section we assume that $\str{K}$ and $L$ are fixed and
satisfy the assumptions of Theorem~\ref{thm:cycles}. 
Following ideas from \cite[Section 4.1]{Hubicka2020CS}, we construct a special $L$-structure $\str G$ with finite big Ramsey degrees and then use $\str G$ to prove 
finiteness of big Ramsey degrees for~$\str K$.

\begin{lemma}\label{lem:strong}
Let $h\colon \str{A}\to \str{B}$ be a homomorphism-embedding. If $\str{B}$ has a completion $c\colon \str{B} \to \str{K}$, then there exists a completion $d\colon \str{A}\to \str{K}$.
\end{lemma}
\begin{proof}
It is clearly enough to consider the case where $c$ is the identity and $h$ is surjective and almost identity, that is, there is a unique vertex $v\in A$ such that $h(v)\neq v$.
Let $\str{B}'$ be the structure induced by $\str{K}$ on $B$.
We create a structure $\str{B}''$ from $\str{B}'$ by duplicating the vertex $h(v)$ to $v'$ and leaving $h(v)$ not adjacent to $v'$. Since $\str{B}'$ is irreducible, it is easy to observe that 
all induced cycles in $\str{B}''$ are already present in $\str{B}'$.
By the assumption on $\str{K}$, there is a completion $c'\colon \str{B}''\to \str{K}$.
Now, the completion $d\colon\str{A}\to \str{K}$ can be constructed by setting $d(v)=c'(v')$ and
$d(u)=c'(u)$ for every $u\in A\setminus \{v\}$.
\end{proof}

We put $\Alphabet=\{\str{A}:A=\{0,1\} \hbox{ and there exists an embedding $\str{A}\to\str{K}$}\}.$
For $U\in \Alphabet^*$, we will use bold characters to refer to the letters (e.g. $\str U_0$ is the structure corresponding to the first letter of $U$) to emphasize that $\Alphabet$ consists of structures.

Given $\str{A},\str{B},\str{C}\in \Alphabet$, there is at most one structure $\str{D}$ with the vertex set $\{u,v,w\}$ satisfying the following three conditions:
(i) mapping $0\mapsto u,1\mapsto v$ is an embedding $\str{A}\to\str{D}$,
(ii) the mapping $0\mapsto v,1\mapsto w$ is an embedding $\str{B}\to\str{D}$, and
(iii) the mapping $0\mapsto u,1\mapsto w$ is an embedding $\str{C}\to\str{D}$.
If such a structure $\str{D}$ exists, we denote it by $\triangle(\str{A},\str{B},\str{C})$ (since  $\str{A},\str{B},\str{C}$ form a triangle). Otherwise we leave $\triangle(\str{A},\str{B},\str{C})$  undefined.

\begin{definition}
Let $\str{G}$ be the following structure.
\begin{enumerate}
  \item\label{itemG1} The vertex set $G$ consists of all finite words $W$ of length at least 1 in the alphabet $\Alphabet$ that satisfy the following condition.
\begin{enumerate}[label=(A\arabic*)]
\item\label{itemGA1} For all $i$ and $j$ with $0\leq i<j<|W|$, the structure induced by $\str{W}_i$ on $\{1\}$ is isomorphic to the structure induced by $\str{W}_j$ on $\{1\}$.
\end{enumerate}
  \item\label{itemG2} Let $U,V$ be vertices of $\str{G}$ with $|U|<|V|$ that satisfy the following condition.
\begin{enumerate}[resume,label=(A\arabic*)]
\item\label{A1} The structure $\triangle(\str{U}_i,\str{V}_{|U|},\str{V}_i)$ is defined for every $i$ with $0\leq i<|U|$ and it has an embedding to $\str{K}$.
\end{enumerate}
Then the mapping $0\mapsto U, 1\mapsto V$ is an embedding of type $\str{V}_{|U|}\to \str{G}$.
  \item\label{itemG3} There are no tuples in the relations $\rel{G}{}$, $\rel{}{}\in L$, other than the ones given by~\ref{itemG2}.
\end{enumerate}
\end{definition}

\begin{lemma}
\label{lem:completion}
Every induced cycle in $\str{G}$ has a completion to $\str{K}$. Since every irreducible substructure of size at most 3 embeds into $\str{K}$ there is a completion $\str{G}\to\str{K}$.
\end{lemma}
\begin{proof}
Suppose for contradiction that there exists $\ell$ and a sequence $U^0$, $U^1$, \ldots, $U^{\ell-1}$ forming an induced cycle
$\str{C}$ in $\str{G}$ such that $\str{C}$ has no completion to $\str{K}$. 
Without loss of generality, we assume that $|U^0|\leq |U^k|$ for every $1\leq k< \ell$.
We create a structure $\str{D}$ from $\str{C}$ by adding precisely those tuples to the relations of~$\str{D}$ such that  
the mapping $0\mapsto \str{U}^0$, $1\mapsto \str{U}^k$ is an embedding from $\str{U}^k_{|U^0|}$ to $\str{D}$ for every $k$ satisfying $2\leq k< \ell$ and $|U^0|<|U^k|$.

For simplicity, consider first the case that we have
$|U^0|<|U^k|$ for every $1\leq k\leq \ell-1$.  In this case, we produced a triangulation of $\str{D}$: all induced
cycles are triangles containing the vertex $U^0$.
It follows from the construction of $\str{G}$ that, for
every $2\leq k\leq \ell$, the triangle induced by $\str{D}$ on $U^0$, $U^k$ and $U^{k+1}$
is isomorphic either to $\triangle(\str{U}^k_{|U^0|},\str{U}^{k+1}_{|U^k|}, \str{U}^{k+1}_{|U^{0}|})$ (if $|U^k|<|U^{k+1}|$) or to
$\triangle(\str{U}^{k+1}_{|U^0|},\str{U}^{k}_{|U^{k+1}|}, \str{U}^{k}_{|U^{0}|})$. 
By~\ref{A1} the triangle has an embedding to $\str{K}$, hence all induced cycles in
$\str{D}$ have a completion to $\str{K}$, which implies that
$\str{D}$ has a completion $c\colon \str{D}\to \str{K}$. We get completion $c:\str{C}\to \str{K}$, a contradiction.

It remains to consider the case that there are multiple vertices of $\str{D}$ of
length $|U^0|$.  We then set $M:=\{U^k: |U^k|=|U^0|\}$.
By the construction of $\str{G}$, the
vertices in $M$ are never neighbours. Moreover, for every $U,V\in M$, the structure induced on $\{U\}$ by $\str{C}$ is isomorphic to structure induced on $\{V\}$
by $\str{C}$, which, by \ref{A1}, is isomorphic to the structure induced on $\{0\}$ by $\str{W}_{|U^0|}$
for every $W\in C\setminus M$.
  Consequently,
it is possible to construct a structure $\str{E}$ from $\str{D}$ by identifying all vertices
in $M$ and to obtain a homomorphism-embedding $f\colon\str{D}\to\str{E}$.
Observe that the structure $\str{E}$ is triangulated and every triangle is known to
have a completion to $\str{K}$.  By Lemma~\ref{lem:strong},  $\str{D}$ also has a completion to $\str{K}$.
\end{proof}
The following result follows directly from the definition of substitution.
\begin{observation}
\label{obs:preserve2}
For every $W\in \Space{\emptyset}{\omega}{\omega}$
and all $U,V\in G$, the structure induced by $\str{G}$ on $\{U,V\}$ is isomorphic to the structure induced by $\str{G}$ on $\{W(U),W(V)\}$.
\end{observation}

Without loss of generality we assume that $K=\omega\setminus \{0\}$. Let $\str{K}'$
be the structure $\str{K}$ extended by the vertex $0$ such that there exists an embedding
$\str{K}'\to\str{K}$. Such a structure $\str{K}'$ exists, because duplicating the vertex 1 does
not introduce new induced cycles.
 We define the mapping $\varphi\colon\omega\setminus\{0\}\to G$ by setting
$\varphi(i)=U$, where $U$ is a word of length $i$ defined
by setting, for every $0\leq j<i$, $\str{U}_j$ as the unique structure in $\Alphabet$ such that $0\mapsto j$, $1\mapsto i$ is an embedding $\str{U}_j\to \str{K}'$.
It is easy to check that $\varphi$ is an embedding $\varphi\colon \str{K}\to \str{G}$.
We prove Theorem~\ref{thm:cycles} in the following form.

\begin{theorem}
\label{thm:cycles2}
For every finite $k\geq 1$ and every finite colouring of subsets of $G$ with $k$ elements, there exists
$f\in \binom{{\str{G}}}{\str{G}}$ such that the colour of every $k$-element subset $S$ of $f(\str{G})$ depends only on
$\tau(S)=\tau(f^{-1}[S])$. 
\end{theorem}

By Proposition~\ref{prop:subspace}, we obtain the desired finite upper bound on the number of colours.
By the completion $c\colon\str{G}\to\str{K}$ given by Lemma~\ref{lem:completion}, the colouring of substructures of $\str{K}$ yields a colouring
of irreducible substructures of $\str{G}$. Embedding $f\in \binom{{\str{G}}}{\str{G}}$ can be restricted $f'\in \binom{{\str{G}}}{\str{K}}$ and gives $c\circ f'\in \binom{{\str{K}}}{\str{K}}$
and thus Theorem~\ref{thm:cycles2} indeed implies Theorem~\ref{thm:cycles}.

\begin{proof}[Sketch]
Fix $k$ and a finite colouring $\chi$ of the
subsets of $G$ of size $k$.   Proposition~\ref{prop:subspace} bounds number of embedding types of subsets of $G$ of size $k$. Apply Theorem~\ref{thm:multCS} for each embedding type. By Observation~\ref{obs:preserve2}, we obtain the desired embedding; see~\cite[proof of Theorem 4.4]{Hubicka2020CS} for details.
\end{proof}

\paragraph{Acknowledgement}
D.~Ch., J.~H., M. K. and J.~N. are supported by the project 21-10775S of  the  Czech  Science Foundation (GA\v CR). This is part of a project that has received funding from the European Research Council (ERC) under the EU Horizon 2020 research and innovation programme (grant agreement No 810115).
M. B. and J. H. were supported by the Center for Foundations of Modern Computer Science (Charles University project UNCE/SCI/004). M. K. was supported by the Charles University Grant Agency (GA UK), project 378119. L.~V. is supported by Beatriu de Pin\'os BP2018, funded by the AGAUR (Government of Catalonia) and by the Horizon 2020 programme No 801370.


\begin{thebibliography}{10}
\providecommand{\url}[1]{{#1}}
\providecommand{\urlprefix}{URL }
\expandafter\ifx\csname urlstyle\endcsname\relax
  \providecommand{\doi}[1]{DOI~\discretionary{}{}{}#1}\else
  \providecommand{\doi}{DOI~\discretionary{}{}{}\begingroup
  \urlstyle{rm}\Url}\fi

\bibitem{Aranda2017}
Aranda, A., Bradley-Williams, D., Hubi{\v c}ka, J., Karamanlis, M.,
  Kompatscher, M., Kone{\v c}n{\'y}, M., Pawliuk, M.: Ramsey expansions of
  metrically homogeneous graphs (2017).
\newblock Submitted, arXiv:1707.02612

\bibitem{Balko2021}
Balko, M., Chodounsk{\' y}, D., Hubi{\v c}ka, J., Kone{\v c}n{\' y}, M., Ne{\v
  s}et{\v r}il, J., Vena, L.: Infinitary {G}raham--{R}othschild theorem for
  words of higher order (2021).
\newblock (in preparation)

\bibitem{Sam}
Braunfeld, S.: Ramsey expansions of {$\Lambda$}-ultrametric spaces (2017).
\newblock ArXiv:1710.01193

\bibitem{carlson1984}
Carlson, T.J., Simpson, S.G.: A dual form of {R}amsey's theorem.
\newblock Advances in Mathematics \textbf{53}(3), 265--290 (1984)

\bibitem{Cherlin2013}
Cherlin, G.: Homogeneous ordered graphs and metrically homogeneous graphs
  (December 2017).
\newblock Submitted

\bibitem{dobrinen2017universal}
Dobrinen, N.: The {R}amsey theory of the universal homogeneous triangle-free
  graph.
\newblock Journal of Mathematical Logic p. 2050012 (2020)

\bibitem{Hubicka2020CS}
Hubi{\v{c}}ka, J.: Big {R}amsey degrees using parameter spaces.
\newblock arXiv:2009.00967  (2020)

\bibitem{Hubicka2016}
Hubi{\v{c}}ka, J., Ne\v{s}et\v{r}il, J.: All those {R}amsey classes ({R}amsey
  classes with closures and forbidden homomorphisms).
\newblock Advances in Mathematics \textbf{356C}, 106,791 (2019)

\bibitem{NVT2009}
Nguyen Van~Th{\'e}, L.: Ramsey degrees of finite ultrametric spaces,
  ultrametric {U}rysohn spaces and dynamics of their isometry groups.
\newblock European Journal of Combinatorics \textbf{30}(4), 934--945 (2009)

\bibitem{The2010}
Nguyen Van~Th{\'e}, L.: Structural {R}amsey Theory of Metric Spaces and
  Topological Dynamics of Isometry Groups.
\newblock Memoirs of the American Mathematical Society. American Mathematical
  Society (2010)

\bibitem{todorcevic2010introduction}
Todorcevic, S.: Introduction to {R}amsey spaces, vol. 174.
\newblock Princeton University Press (2010)

\bibitem{zucker2020}
Zucker, A.: A note on big {R}amsey degrees.
\newblock arXiv:2004.13162  (2020)

\end{thebibliography}
\end{document}